\documentclass{amsart}

\usepackage[colorlinks=true, urlcolor=blue,bookmarks=true,bookmarksopen=true,citecolor=blue,hypertex]{hyperref}
\usepackage{graphicx}
\usepackage{enumerate}

\usepackage{epsfig}
\usepackage{amscd}
\usepackage{amssymb}
\usepackage{amsxtra}
\usepackage{amsmath}
\usepackage{enumerate}
\usepackage{mathrsfs}

\usepackage{color}

\usepackage{amsmath,amsfonts,amssymb}
\usepackage{verbatim}
\usepackage[hmargin = 4cm,vmargin = 2.5cm]{geometry}
\usepackage{epsfig}
\usepackage{amsthm}

\theoremstyle{plain}

\newtheorem{thm}{Theorem}

\newtheorem{lem}[thm]{Lemma}
\newtheorem{prop}[thm]{Proposition}

\theoremstyle{definition}

\theoremstyle{remark}
\newtheorem{rem}[thm]{Remark}

\theoremstyle{plain}

\newcommand\supp{\operatorname{supp}}

\newcommand\im{\operatorname{image}}

\def\ZZ{\mathbb{Z}}

\def\RR{\mathbb{R}}
\def\CC{\mathbb{C}}

\def\dd{\mathrm{d}}
\def\Ham{Ham}

\begin{document}

\pagestyle{headings}

\bibliographystyle{alphanum}

\title{Hamiltonian commutators with large Hofer norm}

\thanks{The author was supported by NSF grant DMS-1105813 and also a Simons instructorship.}

\date{\today} 
\author{Michael Khanevsky}
\address{Michael Khanevsky, Dept. of Mathematics, 
University of Chicago, 5734 S. University Avenue, Chicago, IL 60637, USA}
\email{khanev@math.uchicago.edu}

\begin{abstract}
We show that commutators of Hamiltonian diffeomorphisms may have arbitrarily large Hofer norm. 
The proposed technique is applicable to positive genus surfaces and their products. This gives partial answer to 
the question presented by McDuff and Polterovich in ~\cite{McD:Prblm}.
\end{abstract}

\maketitle

\section{Introduction and results}

In this paper we consider the following question which was presented by D. McDuff and L. Polterovich (question 4 in ~\cite{McD:Prblm}).
Given a closed symplectic manifold $(M, \omega)$, is there an upper bound for Hofer's norm of a commutator
of two Hamiltonian diffeomorphisms of $M$? Clearly, a similar question for $\RR^{2n}$ and the standard symplectic 
form has negative answer: it is sufficient to pick two Hamiltonians $f, g$ with nontrivial commutator $[f, g]$.
Then the norm of the rescaled commutator $\left\|[s f(\frac{\cdot}{s}), s g(\frac{\cdot}{s})] \right\| \to \infty$ as $s$ goes to infinity. However, this rescaling trick cannot be applied 
to closed manifolds and the question becomes not obvious. In this article we construct commutators with arbitrarily large norm in a class of 
closed manifolds which contains surfaces of positive genus and direct products of such surfaces:

\begin{thm} \label{T:main}
  Let $(\Sigma, \omega)$ be a closed surface of positive genus and $(N, \omega')$ be a torus or a closed symplectic manifold which admits a pair of transverse
  closed Lagrangians whose union is weakly exact.
  Then the product manifold $(\Sigma \times N, \omega \oplus \omega')$ admits Hamiltonian commutators with arbitrarily large Hofer norm.
\end{thm}

\medskip

Main tool used in the argument is the theorem below which appeared in ~\cite{Usher:subm-Hnorm}. 
Given two Lagrangians $L, L'$ we define the \emph{separation energy}
\[
  E_{sep} (L, L') = \inf_{g L \cap L' = \emptyset} \| g \|
\]
to be the least Hofer energy needed to move $L$ away from $L'$ by a compactly supported Hamiltonian.
In the case no such deformation exists we set $E_{sep} (L, L') = \infty$. This notion is a natural generalization 
of the displacement energy of a Lagrangian. 

\begin{thm} \label{T:bound}\emph{(M. Usher)}
  Let $(M, \omega)$ be a geometrically bounded symplectic manifold, $L, L' \subset M$ be two compact Lagrangians that 
  intersect transversely. Pick an intersection point $p \in L \cap L'$ and 
  a tame almost complex structure $J$. Then the separation energy satisfies
  \[
   E_{sep} (L, L') \geq \min (\sigma_M, \sigma_L, \sigma_{L'}, \sigma_p)
  \]
  where $\sigma_M$ denotes the minimal energy of a nonconstant $J$-holomorphic sphere in $M$, $\sigma_L, \sigma_{L'}$ stand for the minimal
  energy of nonconstant $J$-holomorphic discs with boundary on $L, L'$, respectively. $\sigma_p$ denotes the minimal energy
  of a nonconstant $J$-holomorphic strip
  \[
	u: (\RR \times [0, 1], i) \to (M, J),
  \]
   with boundary conditions $u (s, 0) \in L$, $u (s, 1) \in L'$ for all $s \in \RR$ and which converges to $p$ either as $s \to +\infty$ or as $s \to -\infty$. 
   Moreover, in the case $J$ is regular we may consider for $\sigma_p$ only those strips that are isolated in the corresponding nonparametrized moduli space.
\end{thm}

Theorem 4.9 in ~\cite{Usher:subm-Hnorm} claims $E_{sep} (L, L') > 0$ which is weaker than the statement above. 
Nevertheless, the argument taken verbatim implies Theorem~\ref{T:bound} except for the last sentence regarding regular 
almost complex structures and isolated holomorphic strips. This additional property can be established by restricting attention in the proof
to holomorphic strips of index zero.

This theorem can be seen as an adaptation of Chekanov's bound ~\cite{Chekanov:Lag-energy-1} for the displacement energy: 
$E_{disp} (L) \geq \min (\sigma_M, \sigma_L)$. As in Chekanov's argument, here is no requirement of existence of 
Floer homology.

Theorem~\ref{T:bound} gives a tool to compute lower bounds for Hofer's norm. Given a Hamiltonian $g$, one may consider
two disjoint Lagrangians $L, L'$ such that $g L$ intersects $L'$ transversely. Then
$\| g \| \geq E_{sep} (g L, L') \geq \min (\sigma_M, \sigma_{g L}, \sigma_{L'}, \sigma_p)$ for all $p \in gL \cap L'$.
In certain cases (like the case when $M$ is a surface) the righthandside expression can be easily computed. 

\medskip

The rest of this paper is organized as follows. Section~\ref{S:def} provides basic definitions for Hofer's geometry and Floer theory 
and in Section~\ref{S:con} we construct commutators for Theorem~\ref{T:main}.

\medskip

\emph{Acknowledgements:}
The author is grateful to L. Polterovich and M. Usher for useful discussions and comments. 
He also thanks the referee for his/her careful work.

\section{Definitions}\label{S:def}

A symplectic manifold $(M, \omega)$ is called \emph{geometrically bounded} (see ~\cite{ALP}) if there exist a complete Riemannian metric $g$ and 
an almost complex structure $J$ such that $(M, g)$ has bounded sectional curvature and injectivity radius bounded away from zero.
In addition we ask that $\omega (v, J v) > c_1 g(v, v)$ and $|\omega (v, w)|^2 < c_2 g(v,v)g(w,w)$ for some constants $c_1, c_2 > 0$.
$J$ is called a \emph{tame} almost complex structure. Note that for closed $M$ it is enough to assume that $\omega (v, Jv) > 0$.
Let $\Sigma$ be a Riemannian surface (possibly with boundary), $u : \Sigma \to M$ a $J$-holomorphic map. 
The \emph{energy} or \emph{symplectic area} of $u$
is defined by $E(u) = \int_{\Sigma} u^* \omega$.

Given two Lagrangian submanifolds $L_1, L_2 \subset M$ that intersect transversely and a tame almost complex structure $J$ we consider the 
set of nonconstant $J$-holomorphic strips
$u:\RR \times [0, 1] \to M$ that satisfy $u(s, 0) \in L_1$, $u (s, 1) \in L_2$ for all $s \in \RR$. Every strip $u$ with $E(u) < \infty$ converges to a pair of 
intersection points $p, q \in L_1 \cap L_2$ as $s \to \pm \infty$. We denote by $\mathcal{M}^* (p, J)$ the set of all 
nonparametrized finite-energy $J$-holomorphic strips as above which converge to a given point $p \in L_1 \cap L_2$ either as $t \to +\infty$ or as $t \to -\infty$. The general theory
(see ~\cite{Fl:Morse-theory}, ~\cite{McD-Sa:Jhol-2}) states that for \emph{regular} almost complex structures $J$, $\mathcal{M}^* (p, J)$ admits a manifold structure. 
The set of regular structures $J$ is generic in the space of all tame almost complex structures on $M$. 

Assume that $(\Sigma, \omega)$ is a closed symplectic surface. Then every tame complex structure $J$ on $\Sigma$ is regular (see Theorem 12.2 in ~\cite{Si-Ro-Sa:CombiFloer}). Let $L_1, L_2$ be two
simple closed connected curves in $\Sigma$ intersecting transversely. Denote $D_+ = \{ z \in \CC \, \big| \, |z| \leq 1 \, , \, \text{Im} (z) \geq 0\}$.
Following ~\cite{Si-Ro-Sa:CombiFloer} we define a \emph{smooth lune} to be an orientation preserving immersion $u : D_+ \to \Sigma$ such that
$u (D_+ \cap \RR) \subset L_1$, $u (D_+ \cap S^1) \subset L_2$. $u(1), u(-1) \in L_1 \cap L_2$ are the endpoints of $u$.
We pick $p \in L_1 \cap L_2$. By Theorem 12.1, ~\cite{Si-Ro-Sa:CombiFloer} there is a bijection between isolated nonparametrized holomorphic strips in $\mathcal{M}^* (p, J)$
and equivalence classes (under reparametrization) of smooth lunes with endpoint at $p$. Moreover, holomorphic strips and corresponding lunes share the same image in $\Sigma$.
In particular, they have the same energy which is equal to the area covered by the image (counted with multiplicity).
In this situation we may replace the definition of $\sigma_p$ in Theorem~\ref{T:bound} by the minimal area of a smooth lune with endpoint at $p$.

\medskip
 We call a closed symplectic manifold $(N, \omega_N)$ \emph{admissible} if it satisfies the following. 
 \begin{itemize}
  \item
	  There exist two Lagrangian submanifolds
	  $K_1, K_2$ that intersect transversely, an intersection point $p \in K_1 \cap K_2$ and a regular tame
	  almost complex structure $J_N$ such that $\mathcal{M}^* (p, J_N)$ has no isolated strips.
  \item
	  All $J_N$-holomorphic spheres and disks with boundary either in $K_1$ or in $K_2$ are constant. (This condition holds automatically
	  when $\omega_N$ vanishes on $\pi_2 (N, K_1)$ and on $\pi_2 (N, K_2)$.)
 \end{itemize}
 Suppose that a closed symplectic manifold $(N, \omega)$ admits a pair of transverse Lagrangians whose union is weakly exact 
 (namely, $\omega$ vanishes on $\pi_2 (N, K_1 \cup K_2)$). 
 It is admissible by the definition above. A closed surface of positive genus is also an example of such manifold:
 we pick $K_1, K_2$ to be a pair of simple closed curves that intersect transversely at single point $p$. Let $J_N$ be an arbitrary complex structure. 
 Let $u$ be a $J_N$-holomorphic strip with bounded energy and boundary on $K_1 \cup K_2$. Its lift $\tilde{u}$ to the universal cover
 maps boundary of the strip $\RR \times [0, 1]$ to lifts $\widetilde{K}_1 \cup \widetilde{K}_2$. 
 The maximum principle implies that boundary of the image of $\tilde{u}$ is contained in $\widetilde{K}_1 \cup \widetilde{K}_2$, hence 
 $\im(\tilde{u}) \subset \widetilde{K}_1 \cup \widetilde{K}_2$. By the open mapping theorem $\tilde{u}$ is 
 a constant map, hence $u \equiv p$. All holomorphic discs and spheres are constant since $\pi_2 (N, K_1) = \pi_2 (N, K_2) = \pi_2 (N) = \{0\}$.
 
 \begin{lem} \label{lem:prod}
	Admissible manifolds are closed under direct product.
 \end{lem}
 \begin{proof}
	Let $N_1, N_2$ be admissible manifolds, $K_{1, N_i}, K_{2, N_i}$ the corresponding Lagrangians.
	Consider the products $K_{j, N_1} \times K_{j, N_2}$, $j=1,2$ of the respective Lagrangians and the product almost complex structure. 
	The regularity of the almost complex structure is achieved by surjectivity of the linearized $\bar{\partial}$ operator (see ~\cite{McD-Sa:Jhol-2}).
	A simple computation shows that surjectivity for the product structure follows from that for $J_{N_1}$ and $J_{N_2}$.
	The product	almost complex structure is tame with respect to the product symplectic form. 
	
	The projections $N_1 \times N_2 \to N_i$, $i = 1, 2$ are $J$-holomorphic,
	therefore every $(J_{N_1}, J_{N_2})$-holomorphic strip in $N_1 \times N_2$ projects to a pair of $J$-holomorphic strips in $N_1$, $N_2$. Conversely,
	given a pair of such strips one may lift them to a strip in $N_1 \times N_2$. 
	This implies that isolated strips in $\mathcal{M}^* \left((p_1, p_2), (J_{N_1}, J_{N_2}); N_1 \times N_2 \right)$ must 
	project to a pair of isolated strips in $\mathcal{M}^* (p_i, J_{N_i}; N_i) \cup \{ \text{const} \}$ which are necessarily constant ones.

	Similarly, $J$-holomorphic spheres and disks project to those in $N_1$ and $N_2$ hence are constant ones.
	The lemma follows.
\end{proof}

By this lemma products of positive genus surfaces are admissible.
In the next section we construct commutators in direct products $M \times N$ where $M$ is a closed 
symplectic surface of positive genus and $N$ is admissible. This proves Theorem~\ref{T:main}.
 
\medskip

Let $(M, \omega)$ be a symplectic manifold, $g$ be a Hamiltonian diffeomorphism with compact support in $M$.
The Hofer norm $\|g\|$ (cf. ~\cite{Hof:TopProp}) is defined by 
\[
  \|g\| = \inf \left\{ \int_0^1 \max G(\cdot, t) - \min G(\cdot, t) \dd t \right\}
\]
where the infimum goes over all compactly supported Hamiltonian functions $G : M \times [0,1] \to \RR$
such that $g$ is the time-1 map of the corresponding flow.

\section{Construction}\label{S:con}

First we construct commutators with large Hofer norm on a torus $T^2$.
We use the following convention: $S^1 = \RR / \ZZ$, $T^2 = S^1 \times S^1 = \RR^2 / \ZZ^2$ equipped with $x, y$
coordinates. Sometimes we will consider $x+iy$ as a single complex coordinate. 
$T^2$ is equipped with the standard symplectic form $\omega = \dd x \wedge \dd y$ so it has area $1$.
Let $\eta': S^1 \to \RR$ be the piecewise linear function given by
\[
 \eta'(s) = \left\{ \begin{array}{ll} 
	4s & \textrm{if $s \in [0, \frac{1}{4}]$}\\
	1 & \textrm{if $s \in [\frac{1}{4}, \frac{1}{2}]$}\\
	3-4s & \textrm{if $s \in [\frac{1}{2}, \frac{3}{4}]$}\\
	0 & \textrm{if $s \in [\frac{3}{4}, 1]$} \end{array}\right.
\]
and $\eta : S^1 \to \RR$ be a $C^0$-close smooth approximation of $\eta'$ given by rounding the four corners in 
their $\varepsilon$-small neighborhoods (see Figure~\ref{F:setup}).

\begin{figure}[!htbp]
\begin{center}
\includegraphics[width=0.8\textwidth]{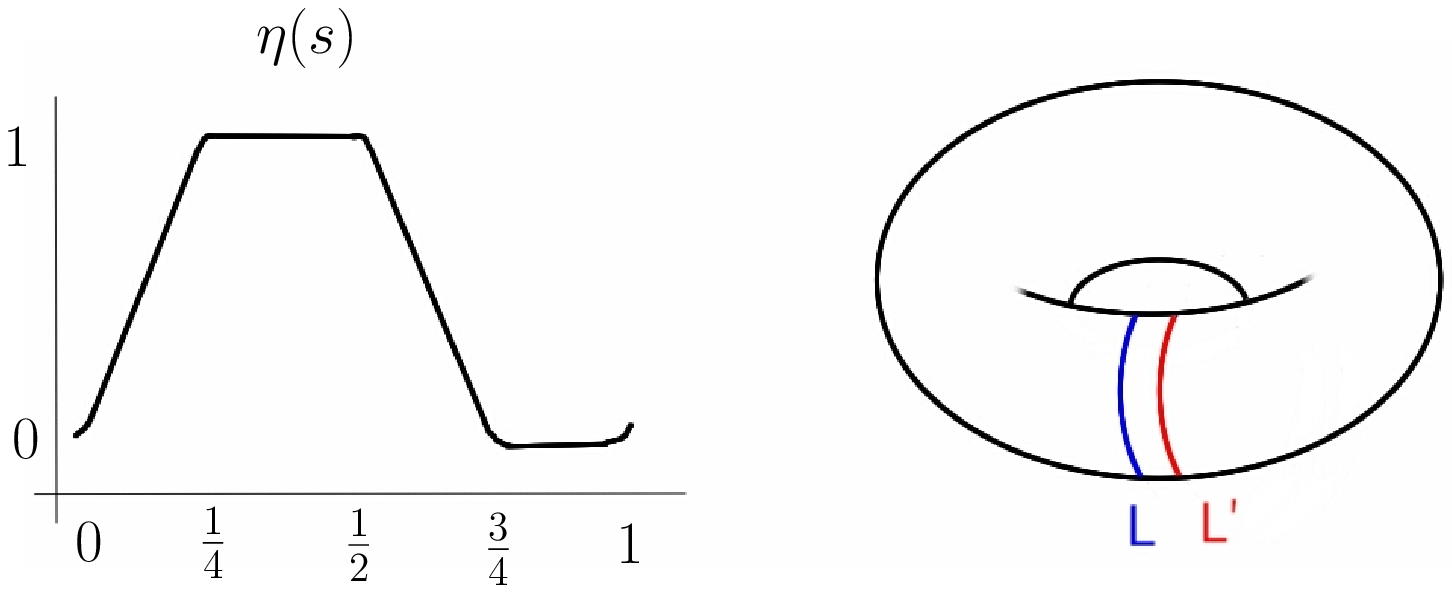}
\caption{}
\label{F:setup}
\end{center}
\end{figure}

We define two Hamiltonian functions $F, G : T^2 \to \RR$ by $F(x, y) = \eta (x)$ and $G (x, y) = \eta (y)$ and denote
by $f_t, g_t$ the time-$t$ maps of the corresponding autonomous flows. Geometrically, $f_t$ rotates upwards all points
in the annulus $-\varepsilon<x<1/4+\varepsilon$ (with the maximal rotation length equal to $4t$), rotates downwards 
the points of the annulus $1/2-\varepsilon<x<3/4+\varepsilon$ and leaves the rest of the torus in place. $g_t$ performs 
the same deformation in the horizontal direction.

\begin{prop}\label{p:torus}
 The Hofer norm of the commutator $[f_t, g_s] = f_{-t} g_{-s} f_t g_s$ satisfies
 \[
	\min(t, s)-1 \leq \left\| [f_t, g_s] \right\| \leq 2 \min (t, s).
 \]

\end{prop}

Clearly, the proposition implies the desired result for $T^2$ by letting $s, t \to \infty$.

\begin{proof}
 The upper bound follows from a standard computation: 
 \[
	\|f_{-t}\| = \|f_t\| \leq \int_0^t \max (F) - \min(F) \dd t = t.
 \]
 By conjugation invariance of Hofer's norm we have $\|g_{-s} f_t g_s\| = \|f_t\| \leq t$ as well. Therefore 
 $\|[f_t, g_s]\| \leq \|f_{-t}\| + \|g_{-s} f_t g_s\| \leq 2 t$ by the triangle inequality. 
 A similar computation shows $\|[f_t, g_s]\| \leq 2 s$.
 
 \medskip
 
 We prove the lower bound. 
 Note that $\|[f_t, g_s]\|$ depends continuously on $s$ and $t$, hence it is enough to show the 
 statement for a dense subset of values.

 \medskip
 
 \underline{Step I:}
 We introduce two Lagrangians meridians 
 \[
  L = \{1-2\varepsilon\} \times S^1, \quad L' = \{1-\varepsilon\} \times S^1. 
 \]
 ($\varepsilon$ is the same as in construction of $\eta$.)
 $L \cap L' = \emptyset$, hence 
 \[
	\|[f_t, g_s]\|  \geq E_{sep} (L, [f_t, g_s] L')
 \]
 and it is enough to show
 that 
 \[
	E_{sep} (L, [f_t, g_s] L') \geq \min(t, s)-1.
 \]
 Bi-invariance of Hofer's norm implies
 \begin{eqnarray} \label{Eq:esep}
	E_{sep} (L, [f_t, g_s] L') &=& E_{sep} (L, f_{-t} g_{-s} f_t g_s (L')) = E_{sep} (g_s f_t (L), f_t g_s (L')) = \\
	&=& E_{sep} (g_s (L), f_t g_s (L')). \nonumber
 \end{eqnarray}
 The last equality follows from the fact that $f_t$ leaves $L$ invariant.
 
 We use Theorem~\ref{T:bound} to get a lower bound for the righthandside expression in (\ref{Eq:esep}). 
 Equip $T^2$ with multiplication by $i$. As 
 \[
  \pi_2 (T^2) = \pi_2 (T^2, g_s (L)) = \pi_2 (T^2, f_t g_s (L')) = \{0\},
 \]
all holomorphic spheres and holomorphic disks with boundary on Lagrangians have zero energy hence are constant ones. Therefore
 \[
	\sigma_{T^2} = \sigma_{g_s (L)} = \sigma_{f_t g_s (L')} = \infty.
 \]
 
 It is enough to pick an intersection point $p$ so that $\sigma_p \geq \min(t, s)-1$. 
 
 \medskip
 
 \underline{Step II:}
 Put $L_s := g_s (L)$, $L_{ts}' :=  f_t g_s (L')$. 
 Consider the universal cover $\pi : \RR^2 \to T^2$. Let
 \[
	\widetilde{L} = \{1-2\varepsilon\} \times \RR, \quad \widetilde{L}' = \{1-\varepsilon\} \times \RR
 \]
 be lifts of $L$ and $L'$, denote by $\tilde{f}_t$, $\tilde{g}_s$ the lifts of $f_t$, $g_s$ to $\RR^2$. 
 $\widetilde{L}_s := \tilde{g}_s (\widetilde{L})$ depicted on Figure~\ref{F:lift} is a lift of $L_s$. It looks like an infinite two-sided comb whose `teeth' 
 have area $s$ each. $\widetilde{L}'_{ts} := \tilde{f}_t \tilde{g}_s (\widetilde{L}')$ is obtained from a similar `comb' 
 $\tilde{g}_s (\widetilde{L}')$ whose
 `teeth' are vertically deformed by $\tilde{f}_t$ in a periodic way. The area bounded between each oscillation of $\widetilde{L}'_{ts}$ and a `tooth'
 of $\widetilde{L}_s$ is in the interval $[t-1, t+1]$.
 The righthand side of Figure~\ref{F:lift} gives a rough description of the half-plane to the left of $\widetilde{L}_{ts}'$. That is, 
 $\widetilde{L}_{ts}'$ is the boundary of the shaded region. Figure~\ref{F:def} gives a
 more detailed description of the intersection pattern of two arcs of $\widetilde{L}_s$ and 
 $\widetilde{L}_{st}'$ (fat lines). In order to simplify the picture some of the remaining parts of 
 curves are hidden, the rest are sketched either by thin or by dotted lines.
  
\begin{figure}[!tbp]
\begin{center}
\includegraphics[width=\textwidth, trim=0 150 100 250]{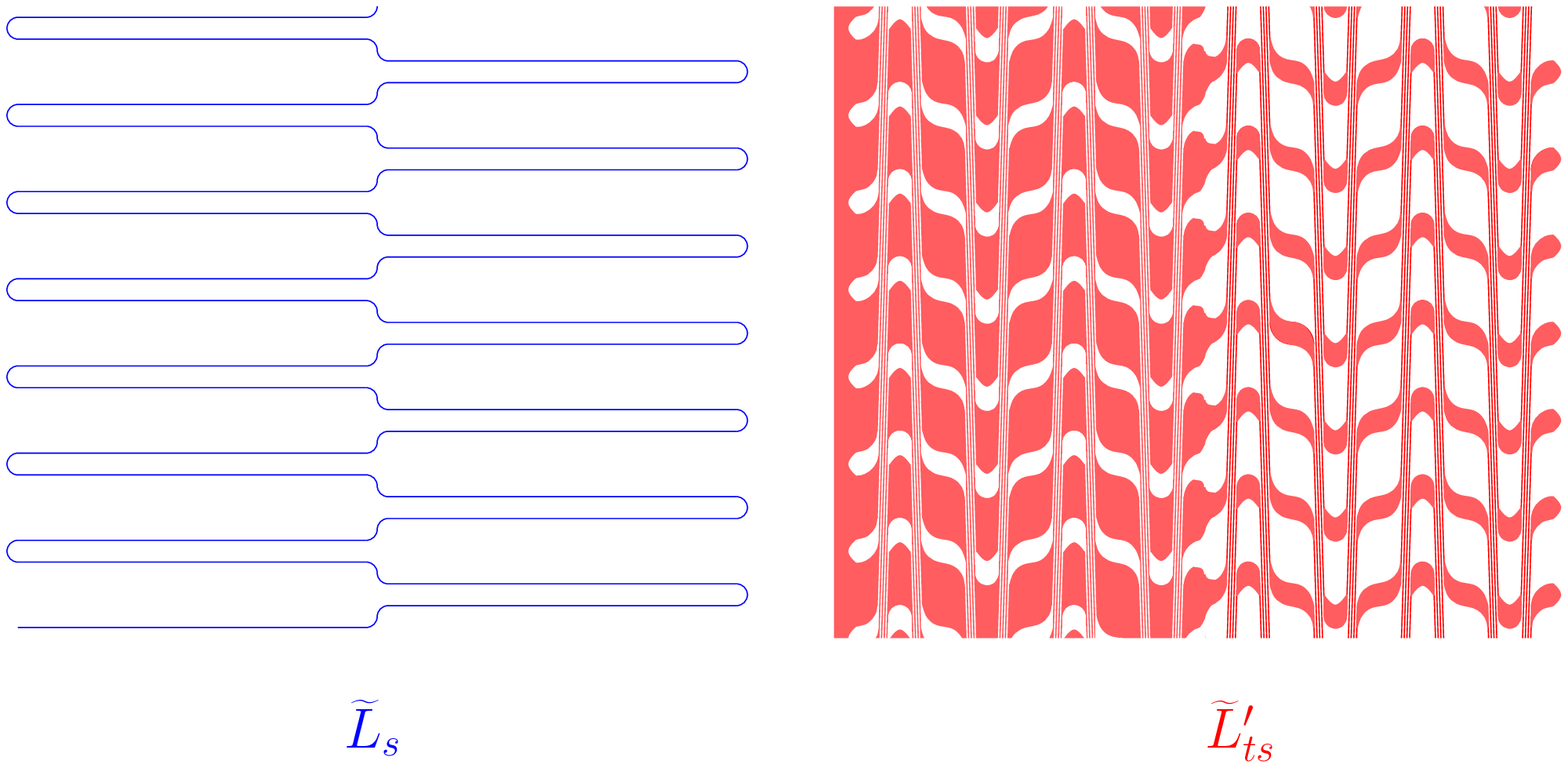}
\caption{}
\label{F:lift}
\end{center}
\end{figure}

\begin{figure}[!htbp]
\begin{center}
\includegraphics[width=0.7\textwidth]{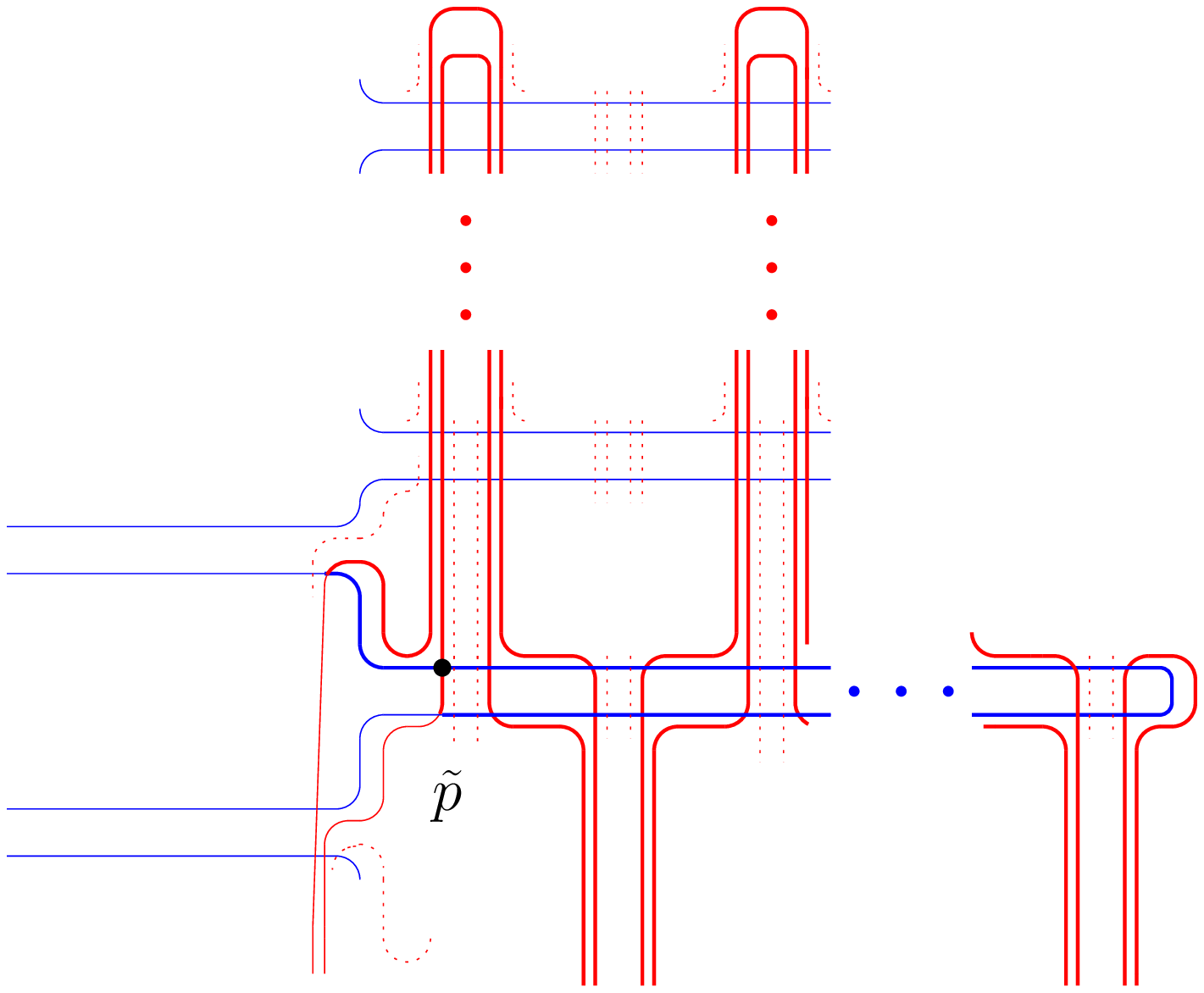}
\caption{}
\label{F:def}
\end{center}
\end{figure}

We may assume that $g_s L \pitchfork f_t g_s (L')$ as this property holds for generic $t$. Pick an intersection point
$\tilde{p} \in \widetilde{L}_s \cap \widetilde{L}_{ts}'$ in the neighborhood
of $\{0\}\times \RR$ as described in Figure~\ref{F:def}. Put $p := \pi (\tilde{p}) \in L_s \cap L_{ts}'$. 

\medskip

\underline{Step III:}
We show that $\sigma_p \geq \min(t, s)-1$. Following the discussion in Section~\ref{S:def}, we may compute $\sigma_p$
by considering the minimal energy of a smooth lune with endpoint at $p$ instead of the energy of holomorphic strips.
Note that every such lune lifts to a smooth lune in $\RR^2$ with endpoint at $\tilde{p}$ and appropriate boundary conditions in $\widetilde{L}_s$ and
$\widetilde{L}_{ts}'$. The lift preserves the energy (area) of a lune, so we proceed with computations on the universal cover instead of $T^2$. 
In Remark 6.11, ~\cite{Si-Ro-Sa:CombiFloer}
the authors propose the following algorithm to locate the lunes. 
Given an intersection point $\tilde{q} \in \widetilde{L}_s \cap \widetilde{L}_{ts}'$, consider the
two arcs $\gamma \subset \widetilde{L}_s$, $\gamma' \subset \widetilde{L}_{ts}'$ connecting $\tilde{p}$ with $\tilde{q}$. $\gamma \cup \gamma'$ bounds 
a smooth lune $u : D_+ \to \RR^2$, $u(\RR \cap D_+) \subseteq \gamma$, $u(S^1 \cap D_+) \subseteq \gamma'$ with $u(-1) = \tilde{p}, u(1) = \tilde{q}$ if and only if the following conditions hold:
\begin{enumerate}
  \item
	Orient $\gamma, \gamma'$ from $\tilde{p}$ to $\tilde{q}$. $\tilde{p}, \tilde{q}$ must have opposite intersection indices. 
  \item
	$\gamma$ must be homotopic to $\gamma'$ relative endpoints.
  \item
	The winding number $w:\RR^2 \setminus (\gamma \cup \gamma') \to \ZZ$ of the closed curve $\gamma * -\gamma'$ must be non-negative
	and satisfy $w(z) \in \{0, 1\}$ near $\tilde{p}$ and $\tilde{q}$.
\end{enumerate}
Moreover, if $\gamma, \gamma'$ satisfy the conditions above, such a lune is unique up to reparametrization.
We may find lunes with $u(-1) = \tilde{q}, u(1) = \tilde{p}$ in a similar way by interchanging the roles of $\tilde{p}$ and $\tilde{q}$.

Figure~\ref{F:lunes} describes six lunes with endpoint at $\tilde{p}$ (the proportions are not precise). The lunes (c-f) cover roughly 
either a straight `tooth' or a vertically deformed one and have energy 
$E(u) \geq s-1$ while the lunes (a) and (b) satisfy $E(u) \geq t-1$. Therefore $E(u) \geq \min(t, s)-1$ holds for all six.

\begin{figure}[!htbp]
\begin{center}
\includegraphics[width=0.6\textwidth]{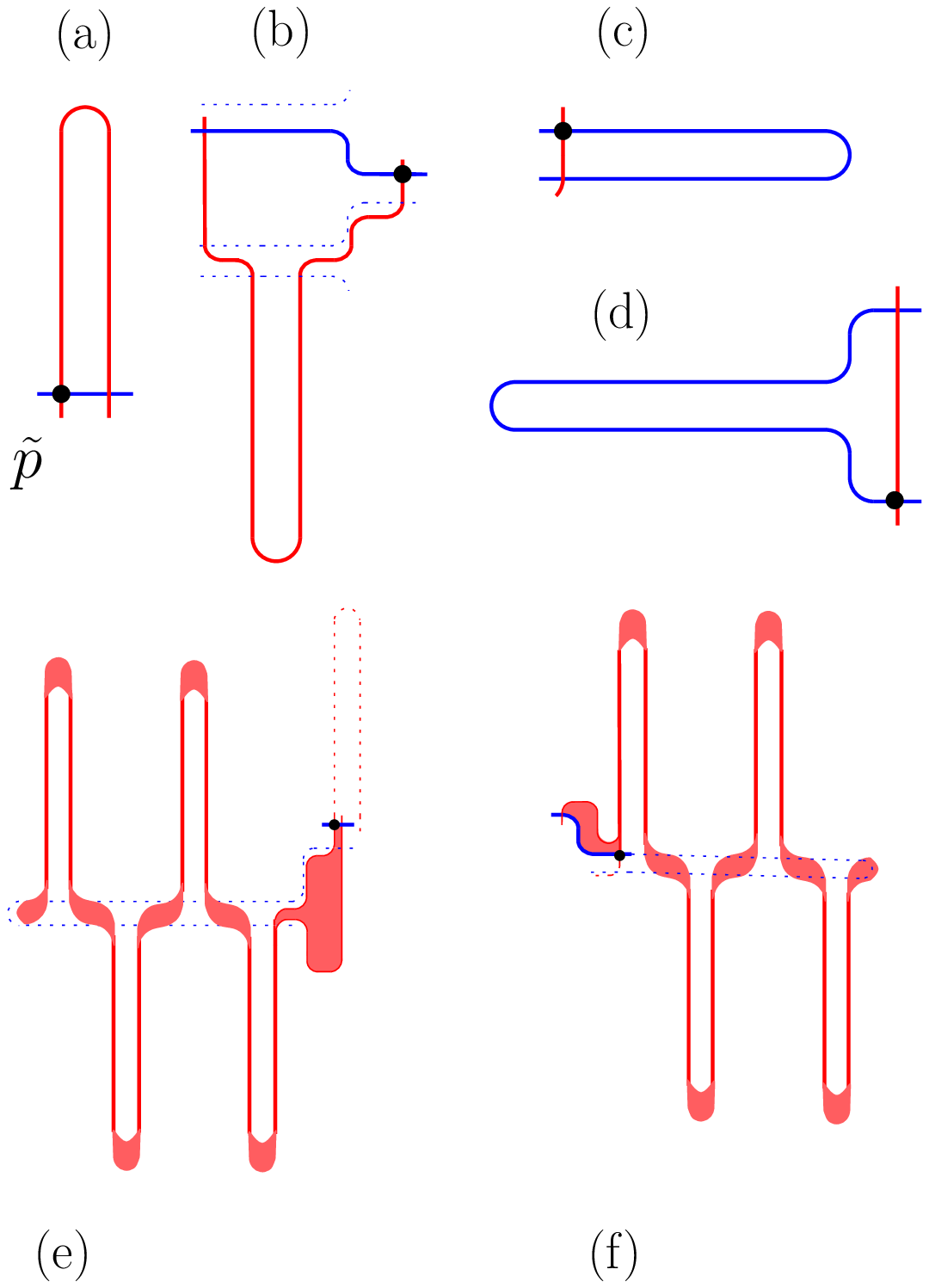}
\caption{}
\label{F:lunes}
\end{center}
\end{figure}

\medskip

\underline{Step IV:}
We show that there are no other lunes. This implies $\sigma_p \geq \min(t, s)-1$ and finishes the proof of the proposition.
Note that in our setup the second condition of the algorithm is always true and the first one holds for many candidate points $\tilde{q}$. 
The only significant constraint is imposed by the third condition.

Let $\tilde{q} \in \widetilde{L}_s \cap \widetilde{L}_{ts}'$ be a candidate for the second endpoint of a lune, 
$\gamma \subset \widetilde{L}_s$, $\gamma' \subset \widetilde{L}_{ts}'$ be the arcs connecting $\tilde{p}$ with $\tilde{q}$.

Observation I: let $j$ be one of (a-f), denote by $u_j$ the corresponding lune. Put $\tilde{q}_j$, $\gamma_j \subset \widetilde{L}_s$, 
$\gamma'_j \subset \widetilde{L}_{ts}'$ to be the second endpoint and boundary arcs of $u_j$.
If $\gamma \supsetneq \gamma_j$ and $\gamma' \supsetneq \gamma'_j$ then $\tilde{q}_j \in \gamma \cap \gamma'$ and the winding number
$w: \RR^2 \setminus (\gamma \cup \gamma') \to \ZZ$ of $\gamma * -\gamma'$ near $\tilde{q}_j$ will attain both positive and negative values. This
is a contradiction to the third condition for existence of a lune in the list above.

Observation II: note that $\gamma$ traverses $\widetilde{L}_s$ which goes from $\tilde{p}$ either to the left or to the right to the 
approximate distance $4s$ and then turns back. Similarly, $\gamma' \subset \widetilde{L}_{ts}'$ and $\widetilde{L}_{ts}'$ goes either up or down to the distance $4t$ and back. 
We call $\gamma$ `short' if $\gamma \subsetneq \gamma_c$ or $\gamma \subsetneq \gamma_d$, otherwise it is `long'.
Similarly, $\gamma'$ is `short' if $\gamma' \subsetneq \gamma'_a$ or $\gamma' \subsetneq \gamma'_b$ and `long' otherwise.
If the curve $\gamma$ is short then $\gamma'$ must be long to arrive to the same endpoint $\tilde{q}$.
And vice versa: a short $\gamma'$ implies that $\gamma$ is long. Therefore at least one of
$\gamma, \gamma'$ is long. That is, either $\gamma$ traverses a full `tooth' or $\gamma'$ goes along a full oscillation.

We assume by contradiction that $\tilde{q}$ is an endpoint which is different from the six endpoints in (a-f). There are eight possible cases:
\begin{itemize}
  \item 
	$\gamma$ is long and goes to the right from $\tilde{p}$ while $\gamma'$ goes down from $\tilde{p}$. 
	Then $\gamma \supsetneq \gamma_c$ hence by observation I, $\gamma' \subset \gamma'_c$.
	But $\gamma'_c$ contains no intersection points with $\widetilde{L}_s$ other than $\tilde{p}$ and $\tilde{q}_c$, a contradiction.
  \item 
	$\gamma$ is long and goes to the right, $\gamma'$ goes up. 
	Note that all points on $\gamma$ have their $y$ coordinate below that of $\tilde{p}$.
	Hence $\gamma'$ must traverse at least one full oscillation to arrive to a point $\tilde{q}$ which is below $\tilde{p}$. 
	Moreover, $\tilde{q}$ is different from $\tilde{q}_a$. But then $\gamma' \supsetneq \gamma'_a$, $\gamma \supsetneq \gamma_a$, a contradiction to observation I.
  \item 
	$\gamma$ is long and goes to the left, $\gamma'$ goes up. 
	Then $\gamma \supsetneq \gamma_d$ hence by observation I, $\gamma' \subset \gamma'_d$.
	But $\gamma'_d$ contains no intersection points with $\widetilde{L}_s$ other than $\tilde{p}$ and $\tilde{q}_d$, a contradiction. 	
  \item 
	$\gamma$ is long and goes to the left, $\gamma'$ goes down. 
	Then all points on $\gamma$ have their $y$ coordinate above that of $\tilde{p}$.
	Hence $\gamma'$ must traverse at least one full oscillation to arrive to $\tilde{q}$ which is above $\tilde{p}$. Moreover,
	$\tilde{q}$ is different from $\tilde{q}_b$.
	Then $\gamma \supsetneq \gamma_b$, $\gamma' \supsetneq \gamma'_b$, a contradiction.
  \item 
	$\gamma'$ is long and goes up from $\tilde{p}$, $\gamma$ goes right.
	Then $\gamma' \supsetneq \gamma_a$ hence by observation I, $\gamma \subset \gamma_a$. 
	Therefore $\tilde{q} \in \gamma_a$. However, on the way up from $\tilde{p}$, $\widetilde{L}_{ts}'$ intersects $\gamma_a$ only at
	the point $\tilde{q}_a$.
  \item 
	$\gamma'$ is long and goes up, $\gamma$ goes left.
	Then $\gamma \supsetneq \gamma_f$ as $\gamma_f$ has no interior intersection points with $\widetilde{L}_{ts}'$.
	We may assume that $\gamma$ is short (the case of a long $\gamma$ was already considered above), namely, $\tilde{q} \in \gamma_d$.
	$\gamma'_f$ intersects $\gamma_d$ at three points other than $\tilde{p}$, two of them are $\tilde{q}_d, \tilde{q}_f$ and the third has the same intersection index as
	$\tilde{p}$, hence cannot be an endpoint of a lune (contradiction to condition (1) of the algorithm).
	Therefore $\gamma'$ must traverse $\gamma'_f$, in contradiction to observation I.
  \item 
	$\gamma'$ is long and goes down, $\gamma$ goes right.
	Then $\gamma \supsetneq \gamma_e$ as $\gamma_e$ has no interior intersection points with $\widetilde{L}_{ts}'$. 
	We may assume that $\gamma$ is short, that is, $\tilde{q} \in \gamma_c$. $\gamma'_e$ intersects $\gamma_c$ at four points:
	$\tilde{p}, \tilde{q}_c, \tilde{q}_e$ and the forth has inappropriate intersection index. Therefore $\gamma'$ must traverse 
	$\gamma'_e$, in contradiction to observation I.
  \item 
	$\gamma'$ is long and goes down, $\gamma$ goes left.
	Then $\gamma' \supsetneq \gamma_b'$. By observation I, $\gamma \subset \gamma_b$. However, on the way down from $\tilde{p}$,
	$\widetilde{L}_{ts}'$ intersects $\gamma_b$ only at the point $\tilde{q}_b$.
\end{itemize}
\end{proof}

\begin{rem}
  $f_1, g_1$ constructed above generate a free group $F_2$ in $\Ham(T^2)$:
  $\supp(F) \cup \supp (G)$ is homotopic to the number eight figure. Pick a common periodic point $x \in T^2$ of period one for both 
  $f_t, g_t$. Then the trajectory of $x$ under the action of $<f_1, g_1>$ gives an isomorphism 
  $<f_1, g_1> \simeq \pi_1 (\supp(F) \cup \supp (G), x)$.
  
  Denote by $S$ the generating set consisting of 
  $f_1, g_1, f_{-1}, g_{-1}$ and their conjugates in $F_2$. D. Calegari observed 
  that the word metric with respect to $S$ satisfies all estimates known to the author of the restriction of Hofer's metric to $F_2$.
  It is easy to show that $\|\cdot\|_S \geq \| \cdot \|_H$. It would be interesting to know if these two metrics are comparable.
\end{rem}

\medskip

We adapt Proposition~\ref{p:torus} to handle closed surfaces $(M, \omega)$ of genus $g > 1$. Without loss of generality
we assume that $Area (M) > 1$. Let $F, G : T^2 \to \RR$ be the Hamiltonian functions from Proposition~\ref{p:torus}.
We present $M$ as a connected sum $T^2 \# \Sigma_{g-1}$ of the torus with a surface of genus $g-1$ where 
$\Sigma_{g-1}$ is glued to $T^2$ along a small circle which does not intersect the supports of $F, G$ and the curves $L$, $L'$. This allows us to push
$F, G, f_t, g_s$ forward to $M$. We continue to denote objects on $M$ with the same notation. We claim that the same bounds on the Hofer norm hold in $M$.
\begin{prop}\label{p:surf}
\[
	\min(t, s)-1 \leq \|[f_t, g_s]\| \leq 2 \min (t, s).
\]
\end{prop}
\begin{proof}
  The proof follows the same lines as Proposition~\ref{p:torus}. We indicate only the necessary changes.
  
  Computation of the upper bound is the same.
  
  \underline{Step I:} we push $L$, $L'$, $p$ from $T^2$ forward to $M$ and equip $M$ with a complex structure which restricts to $i$ on $T^2$.
  As before, we would like to show $E_{sep} (g_s (L), f_t g_s (L')) \geq \min(t, s) - 1$ using Theorem~\ref{T:bound}. 
  We note that $\sigma_{M} = \sigma_{g_s (L)} = \sigma_{f_t g_s (L')} = \infty$.
  
  \underline{Step II:} We work with a cover $\pi: \widehat{M} \to M$ which is obtained by gluing to $\RR^2$ infinitely many copies
  of $\Sigma_{g-1}$ along $\ZZ^2$-periodic lattice. The lifts $\widetilde{L}_s, \widetilde{L}_{ts}', \tilde{p}$ are pushed
  from $\RR^2$ forward to $\widehat{L}_s, \widehat{L}_{ts}', \hat{p}$ in $\widehat{M}$. We get the same picture as in Figure~\ref{F:def} up to copies of $\Sigma_{g-1}$ attached in appropriate places.
  
  \underline{Steps III-IV:}
  We show that $\sigma_p \geq \min(t, s)-1$. We consider smooth lunes with endpoint at $p$ and claim that they arise as pushforward
  of [some of] the lunes (a-f) described in Figure~\ref{F:lunes}. We observe that pushforward is possible only for those lunes that do not cover the attaching circle
  of $\Sigma_{g-1}$. When it is defined, the pushforward preserves the energy, therefore the desired bound for $\sigma_p$ in $M$ follows from that in $T^2$.

  Let $u_M : D_+ \to M$ be a lune with endpoint at $p$, $\hat{u}_M$ be its lift to a lune in $\widehat{M}$ with endpoint at $\hat{p}$. 
  As $\hat{u}_M (\partial D_+) \subset \widehat{L}_s \cup \widehat{L}_{ts}'$, the degree of $\hat{u}_M$ is a locally constant function 
  in $\widehat{M} \setminus (\widehat{L}_s \cup \widehat{L}_{ts}')$. 
  We claim that the degree vanishes in all connected components which contain a copy of $\Sigma_{g-1}$. 
  Indeed, let $\alpha$ be a noncontractible loop in $\Sigma_{g-1}$. If $\hat{u}_M$ has nonzero degree at points of $\alpha$, we lift the picture to the universal cover
  and get a lift of $\alpha$ contained inside the lift of $\hat{u}_M$. However, the lift of $\hat{u}_M$ is a lune hence it is bounded while the lift of $\alpha$ is not bounded, a contradiction.
  
  Therefore the degree of $\hat{u}_M$ is zero in connected components containing copies of $\Sigma_{g-1}$. 
  $\hat{u}_M$ is an orientation preserving immersion by definition of a smooth lune, hence these connected components are outside of the image of
  $\hat{u}_M$. This implies that $\hat{u}_M$ can be obtained by pushing forward a lune in $\RR^2$.
\end{proof}

Let $M = \Sigma \times N$ where $\Sigma$ is a closed surface of positive genus and $N$ is an admissible manifold as defined in Section~\ref{S:def}.
We equip $M$ with a product symplectic form $\omega$. Denote by $\pi_\Sigma : M \to \Sigma$, $\pi_N : M \to N$ the natural projections.
The Hamiltonians $f_t, g_s$ defined in Proposition~\ref{p:surf} lift to $\hat{f}_t = \pi_\Sigma^* f_t$, $\hat{g}_s = \pi_\Sigma^* g_s$. They
satisfy the same inequality in Hofer norm:

\begin{prop}\label{p:prod}
\[
	\min(t, s)-1 \leq \|[\hat{f}_t, \hat{g}_s]\| \leq 2 \min (t, s).
\]
\end{prop}
\begin{proof}
  Computation of the upper bound is the same. 
  
  To prove the lower bound we pick Lagrangians $\widehat{L} := L \times K_1$, $\widehat{L}' := L' \times K_2$
  where $L, L'$ are the Lagrangians in $\Sigma$ defined in Proposition~\ref{p:surf} and $K_1, K_2 \subset N$ are given by the definition of 
  an admissible manifold. We pick a product almost complex structure $\widehat{J} := (i, J_N)$ where $i$ is a 
  complex structure in $\Sigma$ and $J_N$ is as in the definition of an admissible manifold. $\widehat{J}$ is regular by the same argument as
  used in Lemma~\ref{lem:prod}.

  Note that $\widehat{L}_s := \hat{g}_s \widehat{L} = L_s \times K_1$, ($L_s \subset \Sigma$ is the same as in the proof of Proposition~\ref{p:surf}),
  $\widehat{L}_{ts}' := \hat{f}_t \hat{g}_s (\widehat{L}') = L_{ts}' \times K_2$. 
  Let $\hat{p} := (p_\Sigma, p_N)$ where $p_\Sigma \in L_s \cap L_{ts}'$ is as in Proposition~\ref{p:surf} and 
  $p_N \in K_1 \cap K_2$ is taken from the definition of $N$. As before, $\|[\hat{f}_t, \hat{g}_s]\| \geq E_{sep} (\widehat{L}_s, \widehat{L}_{ts}')$ 
  and we apply Theorem~\ref{T:bound} to prove that $E_{sep} (\widehat{L}_s, \widehat{L}_{ts}') \geq \min(t, s) - 1$. 
  
  $\pi_2 (\Sigma) = \{0\}$ together with the second property of admissible manifolds imply that $\widehat{J}$-holomorphic spheres in $M$
  are constant, hence $\sigma_M = \infty$. A similar argument for disks implies $\sigma_{\widehat{L}} = \sigma_{\widehat{L}_{ts}'} = \infty$.

  We note that all $\widehat{J}$-holomorphic strips in $M$ are projected by $\pi_\Sigma$ and $\pi_N$ to pseudo-holomorphic strips in $\Sigma$ and in $N$. 
  Vice versa, given two pseudo-holomorphic strips $u_1$ in $\Sigma$ and $u_2$ in $N$, they lift to $\hat{u} = (u_1, u_2)$ in $M = \Sigma \times N$.
  Isolated $\widehat{J}$-holomorphic strips $\hat{u}$ in $M$ with endpoint at $\hat{p}$ are presented by a pair $(u_1, u_2)$ where $u_1$ is an isolated holomorphic strip
  in $\Sigma$ with endpoint at $p_\Sigma$ and $u_2$ is an isolated $J_N$ holomorphic strip with endpoint at $p_N$. We assume that $\hat{u}$ is not constant. 
  $u_2$ is constant by definition of $N$ while $E(u_1) \geq \min(t, s) - 1$ by computation in Proposition~\ref{p:surf}. We note that 
  $E(u) = E(u_1) + E (u_2) \geq \min(t, s) - 1$ hence $\sigma_{p} \geq \min(t, s) - 1$ and the proposition follows.
\end{proof}

%
%
%

\bibliography{bibliography}

\end{document}